\documentclass[10pt,a4paper]{article}
\usepackage{amsmath,amsthm,amssymb,amsfonts,wasysym}
\usepackage{hyperref}
\usepackage{mathtext}
\usepackage{enumitem}

\newcommand{\R}{{\mathbb{R}}}          

\newcommand{\fraki}{{\mathfrak{i}}}
\newcommand{\XIS}{{\mathfrak{X}}}

\newcommand{\SO}{{\mathrm{SO}}}

\newcommand{\rr}{\rightarrow}
\newcommand{\lrr}{\longrightarrow}

\newcommand{\call}{{\cal L}}             %
\newcommand{\calri}{{{\cal R}^\xi}}             %

\newcommand{\na}{{\nabla}}
\newcommand{\tr}[1]{{\mathrm{tr}}\,{#1}}
\newcommand{\trg}[2]{{\mathrm{tr}_{#1}}\,{#2}}

\newcommand{\End}[1]{{\mathrm{End}}\,{#1}}
\newcommand{\grad}{{\mathrm{grad}}\,}
\newcommand{\dx}{{\mathrm{d}}}
\newcommand{\divg}{{\mathrm{div}}}

\newcommand{\papa}[2]{\frac{\partial#1}{\partial#2}}

\newcommand{\mg}{{\mathrm{g}}}
\newcommand{\Jnagasasa}{J^{\mathrm{NS}}}

\newcommand{\cinf}[1]{{\mathrm{C}}^\infty_{#1}}
\newcommand{\vol}{{\mathrm{vol}}}
\newcommand{\ric}{{\mathrm{Ric}\,}}

\newcommand{\estrela}{{\boldsymbol{\star}}}

\newtheorem{teo}{Theorem}[section]

\newtheorem{coro}{Corollary}[section]
\newtheorem{prop}{Proposition}[section]

\newenvironment{Rema}[1][Remark.]{\begin{trivlist}
\item[\hskip \labelsep {\bfseries #1}]}{\end{trivlist}}

\newenvironment{meuenumerate}
{\begin{enumerate}[label=(\roman*)]
  \setlength{\itemsep}{1pt}
  \setlength{\parskip}{0pt}
  \setlength{\parsep}{0pt}}
{\end{enumerate}}

\def\cyclic{\mathop{\kern0.9ex{{+}
\kern-2.2ex\raise-.28ex\hbox{\Large\hbox{$\circlearrowright$}}}}\limits}

\title{Notes on the Sasaki metric}

\author{R. Albuquerque}

\begin{document}


\maketitle


\

\begin{abstract}

We survey on the geometry of the tangent bundle of a Riemannian manifold, endowed with the classical metric established by S.~Sasaki 60 years ago. Following the results of Sasaki, we try to write and deduce them by different means. Questions of vector fields, mainly those arising from the base, are related as invariants of the classical metric, contact and Hermitian structures. Attention is given to the natural notion of extension or complete lift of a vector field, from the base to the tangent manifold. Few results are original,  but finally new equations of the mirror map are considered.

\end{abstract}


\ 
\vspace*{3mm}
\vspace*{2mm}

\vspace*{7mm}
\vspace*{7mm}
\vspace*{7mm}

\markright{\sl\hfill  R. Albuquerque \hfill}

\section{Introduction}

These notes try to give an informal up-to-date presentation, together with some generalisations and observations, of the fundamental results which one finds in the celebrated article of S.~Sasaki \cite{Sasa} of 1958. The article of Sasaki\footnote{A short biography of S.~Sasaki: \url{http://www-history.mcs.st-andrews.ac.uk/Biographies/Sasaki.html}} studies the differential geometry of tangent bundles of Riemannian manifolds and has been constantly and consistently the reference of many developments of the theory. It is thus a modest commemoration of its 60th anniversary that we bring here. Also we feel it may be interesting to give to light a renewed perspective of the many theorems in the paper, now with some attention on those results which became less known and may be slightly generalised. If not else, we present both a personal and a more invariant approach to those important findings.

Noteworthy is the notion of extension of a vector field, introduced in \cite{Sasa}. It is indeed most natural to the Riemannian geometry of the tangent space. For example, it induces a Sasaki metric Killing vector field from a Killing vector field on the base. We discover other properties not explored neither on that or other articles.

It is irrelevant for the present study, but we wish to remember the reader our interest has mainly in view the construction of \textit{gwistor} space and of a natural exterior differential system of Riemannian geometry introduced in \cite{Alb2019}. Hence the justification, by the negative, of our study keeping aside the geometry of the tangent sphere bundles with Sasaki metric. This is very close, yet more complicated.

\section{The tangent manifold and the Sasaki metric}

\label{Section1}

\subsection{Tangent bundles with a linear connection}

Some general results on vector fields on a given oriented Riemannian smooth manifold $({\cal M},g)$ start the article of Sasaki.

Let $\XIS_U=\Gamma(U;T{\cal M})$ denote the space of vector fields on $U\subset\cal M$.
 
A vector field $X\in\XIS_{\cal M}$ is called {divergence-free} (or {solenoidal} or {incompressible}) if it satisfies $\delta X^\flat=0$. In other words, if the following function, the divergence of $X$, vanishes:
\begin{equation}
 \divg X:=-*\dx*X^\flat=-*\dx(X\lrcorner\vol)=-*\call_X\vol.
\end{equation}

One further proves $\divg X=\delta X^\flat=-\tr{\na_\cdot X}=-\trg{g}{\na_\cdot X^\flat\,\cdot}$, where $\na$ is the Levi-Civita connection.

If $X,Y$ are incompressible, then the Lie bracket $[X,Y]$ is incompressible. This follows from the today well-known identity $\call_{[X,Y]}=[\call_X,\call_Y]$.

$X$ is called a {Killing} vector field if $\call_Xg=0$; this is easy to see to be equivalent to the vanishing, for all $Z,W \in T\cal M$, of 
\begin{equation} \label{Killingvf}
 (\call_Xg)(Z,W)=g(\na_ZX,W)+g(Z,\na_WX) .
\end{equation}

By symmetries, it is immediate that Killing implies incompressible.

A vector field $X$ is called {harmonic} if $\dx X^\flat=0$ and $\divg X=0$.

From now on we let $M$ denote a $\cinf{}$ manifold.

The tangent bundle of $M$ is again a manifold, with charts $(x^i,v^i)$, $i=1,\ldots,m$, on $U\times\R^m$, where $U\subset M$ is open, $(x,U)$ is a chart of $M$ and $m=\dim M$. The transition maps of the vector bundle are induced from a change of charts and their Jacobians\footnote{In the following way: if $(x'^a,v'^a)$ is another chart in a domain $U'$, such that $U\cap U'\neq\emptyset$, then 
\[ v'^a=\papa{x'^a}{x^i}v^i,\qquad\ \dx x^i=\papa{x^i}{x'^a}\dx x'^a, \qquad\
   \dx v^i=v'^a\frac{\partial^2x^i}{\partial x'^a\partial x'^b}\dx x'^b+\papa{x^i}{x'^a}\dx v'^a \] and
\[ \papa{ }{x'^a}=\papa{x^i}{x'^a}\papa{ }{x^i}+v'^b\papa{^2x^j}{x'^a\partial x'^b}\papa{ }{v^j},\qquad\qquad 
\papa{ }{v'^a}=\papa{x^i}{x'^a}\papa{ }{v^i} . \] }. Almost tautological is the assertion that we have a manifold
\begin{equation}
 TM=\bigl\{u\in T_xM:\ x\in M\bigr\}
\end{equation}
with embedded linear fibres $T_xM$ if and only if we have a vector bundle $\pi:TM\rr M$ over $M$ with the same linear fibres. The obvious bundle projection is denoted by $\pi$.

Now let us suppose $M$ is endowed with a linear connection, that is, a covariant derivative or local operator $\na:\Gamma(U;TM)\lrr\Gamma(U;T^*M\otimes TM)$ on vector fields on $U\subset M$ satisfying Leibniz rule. We have already used the Levi-Civita connection, which is metric, $\na g=0$, and {torsion-free}, i.e. the tensor  $T^\na(X,Y)=\na_XY-\na_YX-[X,Y]=0$. For the moment, let us consider any linear connection on $M$.

In charts, we have the Christoffel `symbols' as the coefficients in $\na_i\partial_j=\Gamma_{ij}^k\partial_k$.

Every $X\in\XIS_M$ \textit{lifts} both to a \textit{horizontal} $\pi^*X$ and a \textit{vertical} vector field $\pi^\estrela X$ over the tangent manifold. Indeed both belonging to $\XIS_{TM}$. Distinguished in the chart $(x^i,v^i)$, we have
\begin{equation}
 \pi^*\partial_i=\partial_i-v^j\Gamma_{ij}^b\partial_{v^b}\qquad \mbox{and}\qquad \pi^\estrela\partial_i=\partial_{v^i}.
\end{equation}

Hence, if $X=X^i\partial_i$, then 
\begin{equation} \label{vectorlifts}
 \pi^*X=X^i\pi^*\partial_i
 =X^i(\partial_i-v^j\Gamma_{ij}^b\pi^\estrela\partial_b)\qquad\ \ \mbox{and}\qquad   \ \ \pi^\estrela X=X^i\pi^\estrela\partial_i
\end{equation}

It is very often that one omits the pull-back notation when speaking of functions on $TM$ arising from $M$.

\subsection{New tensors on the tangent manifold}

The above information may be recovered from the covariant derivative of a canonical vector field $\xi\in\XIS_{TM}$; namely, for any $u\in TM$, 
\begin{equation}
 \xi_u=u \:\in\pi^\estrela TM .
\end{equation}
It is not a lift, yet a vector field independent of $\na$; locally $\xi=v^i\pi^\estrela\partial_{i}$.

A splitting of $TTM=H\oplus V$ then arises from a short exact sequence over the manifold $TM$:
\begin{equation}
 0\lrr V\lrr TTM\stackrel{\dx\pi}{\lrr}\pi^*TM\lrr0
\end{equation}
with the \textit{vertical} and \textit{horizontal} subbundles given by 
\begin{equation}
 V=\pi^\estrela TM,\qquad
 H=\ker\bigl(\pi^\estrela\na_\cdot\xi\bigr)\qquad\mbox{and}\qquad  \pi^\estrela\na_{\pi^\estrela X}\xi=\pi^\estrela X .
\end{equation}
The connection induces vector bundle projections, usually denoted $(\cdot)^h$ and $(\cdot)^v$, allowing us to identify $H$ with $\pi^*TM$ via $\dx\pi_{|H}$. Further, the latter isomorphism yields instantly that a horizontal lift is global and well-defined. Of course, this may be proved directly.

Notice we use $(\cdot)^*$ for horizontal lifts and $(\cdot)^\estrela$ for vertical lifts. 

The connection on $M$ is furthermore pulled-back to both pull-back bundles. In a synthesised notation, we then obtain a linear connection on the tangent manifold, respecting $TTM=H\oplus V$, yet always denoted $\na^*$. It is the connection-sum
\begin{equation}
 \na^*=\pi^*\na\oplus\pi^\estrela\na  .
\end{equation}

The torsion of $\na^*$ is immediately found by applying the two projections, also knowing it is a tensor:
\begin{equation} \label{torsionNablaasterisco}
 T^{\na^*}=\pi^*T^\na\oplus R^{\pi^\estrela\na}(\cdot,\cdot)\xi .
\end{equation}
Since this is a tensor we may use lifts in the proof. Furthermore we have that 
\begin{equation}
 \calri(X,Y):=R^{\pi^\estrela\na}(X,Y)\xi={\pi^\estrela R^\na}(X,Y)\xi
\end{equation}
only depends on the horizontal components of $X,Y\in\XIS_{TM}$. Of course, $R^\na(X,Y)Z$ on $M$ denotes the curvature $(3,1)$-tensor, defined by $\na_X\na_YZ-\na_Y\na_XZ-\na_{[X,Y]}Z$.

A first consequence of the above is the well-known result that $H$ defines an involutive distribution if and only if $R\equiv0$, i.e. $(M,g)$ is flat. 

In our studies of tangent bundle structures, we have found it most useful to define what we call the \textit{mirror} map $B:TTM\lrr TTM$, indeed an endomorphism of the tangent bundle of the tangent manifold:
\begin{equation}
 B\pi^*X=\pi^\estrela X,\qquad \quad B\pi^\estrela X=0 .
\end{equation}

With the mirror map and its formal adjoint we define the \textit{Nagano-Sasaki} almost-complex structure\footnote{This structure is commonly attributed to S.~Sasaki, however, in probably more acquainted references, it is due to T.~Nagano \cite{Nagano}, cf. \cite{OkuTacha} and others.} on $TM$:
\begin{equation}
 \Jnagasasa=B-B^\mathrm{t} .
\end{equation}
In particular, $TM$ is always an orientable space. An easy computation yields $\Jnagasasa$ integrable if and only if $T^\na=0$ and $R^\na=0$, a result first proved in \cite{Nagano} (cf. \cite{OkuTacha}).

Recently, so-called golden structures have been defined and a particularly natural example on tangent bundles is discovered in \cite{CenGezSal}. It is easily seen the following map $\varphi\in\End{TTM}$ satisfies $\varphi^2-\varphi-I=0$:
\begin{equation}
 \varphi=\frac{1}{2}I+\frac{\sqrt{5}}{2}(B+B^\mathrm{t}).
\end{equation}

The endomorphisms $B$, $\Jnagasasa$ and $\varphi$ are all parallel for $\na^*$, the main identity being
\begin{equation}
 \na^*B=0.
\end{equation}

We may further define a canonical vector field on the tangent manifold. Globally, by
\begin{equation}
 S=B^\mathrm{t}\xi
\end{equation}
and locally by $S=v^a\pi^*\partial_a$. This is called the \textit{geodesic flow} or \textit{geodesic spray} and clearly satisfies
\begin{equation}
 \na^*_{\pi^*X}S=0,\qquad\quad\na^*_{\pi^\estrela X}S=\pi^*X.
\end{equation}
In other words, $B^\mathrm{t}=\na^*_\cdot S$.

Given $X\in\XIS_M$, one may define the \textit{extended} vector field or the \textit{extension} or yet the \textit{comp\-lete lift}\footnote{Apparently it was Sasaki who first saw the relevance of such older concept within tangent manifold geometry. In this article, we conform to the term `extended', rather than the more commonly adopted `complete lift'.} of $X$ on the manifold $TM$:
\begin{equation} \label{extensionvfi}
 \widetilde{X}=X^i\partial_i+v^j\papa{X^i}{x^j}\partial_{v^i}.
\end{equation}
With any $\na$ torsion-free linear connection ($\Gamma_{ij}^k=\Gamma_{ji}^k$) we may write globally, inspired by \cite{Tanno}:
\begin{equation} \label{extensionvfie}
 \widetilde{X}=\pi^*X+\na^*_{S}\pi^\estrela X.
\end{equation}
Indeed, $\widetilde{X}=X^i\partial_i-X^jv^a\Gamma_{ja}^b\pi^\estrela\partial_b+v^a(\partial_aX^j)\pi^\estrela\partial_j+v^aX^j\Gamma_{aj}^b\pi^\estrela\partial_b$.

The geodesic flow $S=B^\mathrm{t}\xi$ is the vector field over $TM$ whose integral parameterised curves $\tau$ are the velocities $\dot{\gamma}$ of geodesics $\gamma$ of $(M,\na)$. Since $S_\tau=(B^\mathrm{t}\xi)_\tau=B^\mathrm{t}\tau\in H$, we can see the former condition is $B^\mathrm{t}\tau=\dot{\tau}$ if and only if $\tau=\dot{\gamma}$ for some curve $\gamma$ in $M$ such that $\na^*_{\dot{\tau}}\xi=\gamma^*\na\dot{\gamma}=0$. In other words, a curve $\gamma$ in $M$ is a geodesic if and only if $(\dx\pi_{|H})^{-1}(\dot{\gamma})=\ddot{\gamma}$.

Another concept is that of \textit{fibre-preserving} vector field $Z$ over $TM$, i.e a vector field such that the induced transformations preserve the fibres $T_xM$, for all $x$. It must verify $\call_ZY\in\Gamma(V)$, for all $Y\in\Gamma(V)$. With $T^\na=0$, it is equivalent to $Z^h$ depending only of $x$. This is immediate from \eqref{torsionNablaasterisco}.

\subsection{Sasaki metric}

If $(M,g)$ is a Riemannian manifold, then we have an associated canonical torsion-free metric connection $\na$ on $M$. From now on $\na$ is the Levi-Civita connection and $R$ denotes its curvature tensor.

With such connection one defines the Sasaki metric on the manifold $TM$:
\begin{equation}
 \mg:=\pi^*g\oplus\pi^\estrela g.
\end{equation}
It follows immediately that $\na^*\mg=0$.

The golden structure $\varphi$ is compatible with the metric, i.e.  $\varphi\lrcorner\mg$ is again a symmetric tensor.

Since $\Jnagasasa$ is compatible with the metric and $\na^*\Jnagasasa=0$, the associated Hermitian structure is $\na^*$-parallel. Writing $\omega=-\Jnagasasa\lrcorner\mg$, we find
\begin{equation}
\begin{split}
 \dx\omega(X,Y,Z) &=  -\omega(\calri(X,Y),Z)
             -\omega(\calri(Y,Z),X)-\omega(\calri(Z,X),Y) \\
     & = \mg(\calri(X,Y),BZ)+\mg(\calri(Y,Z),BX)+\mg(\calri(Z,X),BY) \:=\:0.
\end{split}
\end{equation}
Hence $\dx\omega=0$ follows from Bianchi identity.

The symplectic 2-form $\omega$ on $TM$ is actually the pull-back of the canonical symplectic 2-form $\dx\lambda$ on the cotangent manifold, via the diffeomorphism $\ell=\cdot^\flat:TM\lrr T^*M$, $\ell(u)=g(u,\cdot)$. Such intrinsic structure arises from the Liouville 1-form $\lambda$, giving $\ell^*\lambda=S^\flat$ (the Sasaki dual). Indeed $\dx\ell$ is linear on the vertical directions and preserves the horizontal subspaces for the dual connections on $TM$ and $T^*M$. Hence
\begin{equation}
\begin{split}
(\ell^*\lambda)_u=\lambda_{\ell(u)}\circ\dx\ell=g(u,\dx\pi(\dx\ell(\cdot)))= \qquad\qquad \quad\quad &  \\
=g(u,\dx\pi(\cdot))= \mg(S_u,\cdot)=\theta_u .&
\end{split}
\end{equation}

On the other hand, writing $\theta=S^\flat$ and recurring to the torsion-free connection
\begin{equation}
D^*=\na^*-\frac{1}{2}\calri ,
\end{equation}
we easily prove that, for all $X,Y\in TTM$,
\begin{equation}
 \dx\theta(X,Y)=\mg(BY,X)-\mg(BX,Y) .
 \end{equation}
Thus $\omega=\dx\theta=\ell^*\dx\lambda$.

Still making use of $D^*$, we find
\begin{equation}
 \dx\xi^\flat=0 .
\end{equation}
Since (cf. \cite{Alb2019})
\begin{equation}
 *\theta=\pm\xi^\flat\wedge(\dx\theta)^{m-1},
\end{equation}
we deduce the result that $S$ is incompressible: $\delta\theta=0$.

The Levi-Civita connection $\na^\mg$ of $\mg$ is said to be found in \cite{Sasa}. Nevertheless, we have a most simple and useful expression for $\na^\mg$: for all $X,Y\in\XIS_{TM}$,
\begin{equation}
 \na^\mg_XY=\na^*_XY+A(X,Y)-\frac{1}{2}\calri(X,Y)
\end{equation}
where $A(X,Y)$ is the symmetric term which makes the connection metric. It is defined by
\begin{equation}
 \mg(A(X,Y),Z)=\frac{1}{2}(\mg(\calri(X,Z),Y)+\mg(\calri(Y,Z),X)).
\end{equation}
For instance, $\na^\mg_{\pi^\estrela X}\pi^\estrela Y=0$, for all $X,Y\in\XIS_M$.

It is very important to observe that $A$ takes only horizontal values, because $\calri$ vanishes on any vertical direction, and $\calri$ takes only vertical values.

Henceforth the fibres $T_xM$ are totally geodesic, for all $x\in M$.

The Riemannian curvature tensor of $\mg$ is first computed by O.~Kowalski in \cite{Kow1}. Also Kowalski finds that $\mg$ is locally symmetric if and only if $R=0$; which, in turn, implies that $R^\mg=0$.

The local holonomy algebra of the Sasaki metric is $\R$-linearly generated by three types of tensors, now following \cite[p.146]{Alb2014d}. These are curvature skew-adjoint operators $R^\mg_o(X,Y)$ at the zero-section $o$, given, with respect to $TTM=H\oplus V$, by
\begin{equation}
\begin{split}
 & \hspace{15mm} \left[\begin{array}{cc} 
\pi^*R(X^h,Y^h) & 0 \\ 0 & \pi^\estrela R(X^h,Y^h)
      \end{array}\right] , \\
 &\left[\begin{array}{cc} 
  0 &-\frac{1}{2}\mg(\pi^\estrela R(X^h,\ )Y^v,\ ) \\ \frac{1}{2}\mg(\pi^\estrela R(X^h,\ )Y^v,\ )^\dagger & 0
     \end{array}\right] , \\
  &  \hspace{20mm} \left[\begin{array}{cc} 
\mg(\pi^\estrela R(\ ,\ )X^v,Y^v) & 0 \\ 0 & 0
      \end{array}\right].
\end{split}
\end{equation}

The last equation yields immediately $\Jnagasasa$ integrable if and only if $R=0$.

The holonomy algebra is the Lie algebra of the holonomy group and, by the Theorem of Ambrose-Singer, it is generated by all curvature tensors at \textit{all} points of the manifold. On the other hand all holonomy groups are conjugate to each other, for $TM$ connected i.e. $M$ connected. So if we know the holonomy at the zero-section, we know a subgroup of the global holonomy.

Let us see an example. Suppose $M$ connected has constant sectional curvature $c\neq0$. Then the holonomy group of $(TM,\mg)$ is $\SO(2m)$. Indeed, the maps $\mg(\pi^\estrela R(X^h,\ )Y^v,\ )=c\mg(X^h,\ )\mg(Y^v,\ )$ generate an $m^2$ dimensional subspace, therefore we deduce the total dimension of the three linearly independent subspaces is $m(m-1)/2+m^2+m(m-1)/2=m(2m-1)$.

The Riemannian manifold $(TM,\mg)$ is Einstein if and only if $R=0$, i.e. $M$ is flat. Indeed, we have for instance from \cite[Proposition 1.3]{Alb2011} that
\begin{equation}
\mathrm{Scal}^\mg=\mathrm{Scal}^g-\frac{1}{4}\|\calri\|^2
\end{equation}
We recall, $\|\calri\|^2=v^{p}v^{p'}R_{kpij}R_{k'p'i'j'}g^{ii'}g^{jj'}g^{kk'}$ where $R_{kpij}=g(R(\partial_{i},\partial_{j})\partial_{p},\partial_{k})$.

Many variations of the Sasaki metric have been defined and developed, specially those with weights. We refer the reader to \cite{AbbCalvaPerrone,Alb2011,Alb2014a,Calva,KMS,KowSek0} and the references there-in, concerning the so-called $g$-natural metrics.

Recently, the author discovered \textit{ciconia metric} on any $TM^2$ which is truly natural to any oriented Riemannian surface $M^2$, cf. \cite{Alb2018}. Ciconia metric is not present in the classification of $g$-natural metrics.

\section{Further on Sasaki's results}

\subsection{Theorems of Sasaki on vector fields}

Let us now see the theorems of Sasaki on vector fields on $M$ and $TM$. We notice that one may complement the results by further considering some horizontal lifts or their duals.

Every vertical lift $\pi^\estrela X$ is an incompressible field. One uses the torsion-free connection $D^*=\na^*-\frac{1}{2}\calri$ and well-known formulas for exterior derivative to prove
\begin{equation}
 \call_{\pi^\estrela X}(\pi^*\vol\wedge\pi^\estrela\vol)=
 \pi^*\vol\wedge\dx({\pi^\estrela X}\lrcorner\pi^\estrela\vol)=0.
\end{equation}
This is also clear, using symmetries of $\calri$ and $g$:
\begin{equation} \label{divergenceofverticalliftofvectorfields}
 \begin{split}
 \tr{\na^\mg_\cdot\pi^\estrela X} 
  &=\mg(\na^\mg_{\pi^*\partial_i}\pi^\estrela X,\pi^*\partial_j)g^{ij}+\mg(\na^\mg_{\pi^\estrela\partial_i}\pi^\estrela X,\pi^\estrela\partial_j)g^{ij}    \\
  &=\mg(\na^*_{\pi^*\partial_i}\pi^\estrela X,\pi^*\partial_j)g^{ij}+
  \frac{1}{2}\mg(\calri_{\pi^*\partial_i,\pi^*\partial_j},\pi^\estrela X)g^{ij}\:=\: 0.
 \end{split}
\end{equation}

On the other hand, we deduce
\begin{equation} \label{derivativeofverticalliftdual}
 \dx(\pi^\estrela X)^\flat(Y,Z)=\mg(\na^*_Y\pi^\estrela X,Z)-\mg(\na^*_Z\pi^\estrela X,Y)+\mg(\pi^\estrela X,\calri(Y,Z)) .
\end{equation}
Therefore, by analysis of three cases for $Y,Z\in H\oplus V$, we conclude that a vertical lift $\pi^\estrela X$ is harmonic if and only if $\na X=0$.

Given $X\in\XIS_M$, the extended vector field $\widetilde{X}$ is incompressible if and only if the horizontal lift $\pi^*X$ is so. Indeed we have
\begin{equation}
  \tr{\na^\mg_\cdot(\na^*_S\pi^\estrela X}) 
  =\tr{\na^\mg_\cdot\pi^*X}=\tr{\na_\cdot X} .
\end{equation}
Let us see the first identity:
\begin{equation}\label{remarkablecomputation}
 \begin{split}
 & \ \ \ \mg(\na^\mg_{\pi^*\partial_i}\na^*_S\pi^\estrela X,\pi^*\partial_j)g^{ij}+\mg(\na^\mg_{\pi^\estrela\partial_i}\na^*_S\pi^\estrela X,\pi^\estrela\partial_j)g^{ij}  \ = \\
 & =\mg(\na^*_{\pi^*\partial_i}\na^*_S\pi^\estrela X,\pi^*\partial_j)g^{ij} +\frac{1}{2}\mg(\calri(\pi^*\partial_i,\pi^*\partial_j),\na^*_S\pi^\estrela X)g^{ij} + \\  &\ \ \qquad\qquad +\mg(\na^*_{\pi^\estrela\partial_i}(v^k\na^*_{\pi^*\partial_k}\pi^\estrela X),\pi^\estrela\partial_j)g^{ij}\\
 & =(\mg(\na^*_{\pi^*\partial_i}\pi^\estrela X,\pi^\estrela\partial_j)+ \mg(v^k\na^*_{\pi^\estrela\partial_i}\pi^\estrela(\na_{\partial_k}X),\pi^\estrela\partial_j))g^{ij}\\
 & =g(\na_{\partial_i}X,\partial_j)g^{ij}
 \end{split}
\end{equation}
Hence, cf. \cite[Theorem 6]{Sasa},
\begin{equation}
 \divg\,\widetilde{X}=\divg(\pi^*X+\na^*_{S}\pi^\estrela X)=2\divg\,X .
\end{equation}
In particular, $\widetilde{X}$ is incompressible if and only if $X$ is incompressible.
\begin{Rema} \label{remarkveryuseful}
 The above computation is quite useful. For any vector fields $X,Y\in\XIS_M$, we have 
 \begin{equation}
  \na^*_{\pi^\estrela Y}\na^*_S\pi^*X=\na^*_{\pi^*Y}\pi^*X=
  \pi^*(\na_{Y}X) \qquad\mbox{and}\qquad
  \na^*_{\pi^\estrela Y}\na^*_S\pi^\estrela X=\pi^\estrela(\na_YX)  .
 \end{equation}
Moreover, we find $\na^*_{\pi^*\partial_i}\na^*_S\pi^\estrela X=v^b\pi^\estrela(\na_i\na_bX-\na_{\na_i\partial_b}X)=v^b\pi^\estrela\na^2X(\partial_i,\partial_b)$. And the same is valid with $\pi^*X$ in analogous form. E.g. we may write $\na^*_{Y^h}\na^*_S\pi^*X=\pi^*\na^2X(Y^h,S)$.
\end{Rema}

Just as in \eqref{derivativeofverticalliftdual}, we find as an equivalent condition for $(\widetilde{X})^\flat$ to be closed, that $X^\flat$ is closed, $\na^2X=0$ and $\na_WX\perp R^\na(Y,Z)W$, for all $Y,Z,W\in TM$. This is quite a strong set of equations, which is deduced from the three cases analysis, with $Z,Y\in H\oplus V$, of
\begin{equation}
 \dx(\widetilde{X})^\flat(Y,Z)=\mg(\na^*_Y\widetilde{X},Z)-\mg(\na^*_Z\widetilde{X},Y)+\mg(\na^*_S\pi^\estrela X,\calri(Y,Z)) .
\end{equation}

Regarding horizontal lifts, we have just seen that $\pi^*X$ is incompressible if and only if $X$ is incompressible. Now we find
\begin{equation}
 \dx(\pi^*X)^\flat(Y,Z)=\mg(\na^*_Y\pi^*X,Z)-\mg(\na^*_Z\pi^*X,Y) =\pi^*(\dx X^\flat)(Y,Z) .
\end{equation}
and thus $(\pi^*X)^\flat$ is closed if and only if $X^\flat$ is closed. In other words, if and only if $X$ is locally a gradient, because $(\pi^*X)^\flat=\pi^*X^\flat$.

We conclude that $\pi^*X$ is harmonic if and only if $X$ is harmonic.

\subsection{Theorems of Sasaki on 1-forms}

It is clear that if the pull-back of 1-forms $\pi^*\mu=\dx\Phi$ is a gradient co-vector of a function on $TM$, then such function $\Phi$ is constant along the fibres and then $\mu$ is a gradient co-vector. Reciprocally, if $\mu=\dx f$, then $\pi^*\mu=\dx(f\circ\pi)$.

Also  computations yield $\delta\pi^*\alpha=\pi^*\delta\alpha$ for any 1-form on $M$. Moreover, $\pi^*\alpha$ is harmonic if and only if $\alpha$ is harmonic.

Just as with vertical lift of contra-variant vector fields, the covariant field
\begin{equation}
\pi^\estrela\alpha=\pi^*\alpha\circ B^\mathrm{t}
 \end{equation}
is always incompressible.

To prove the last two assertions we recall adapted frames, which here are defined as $\{e_1,\ldots,e_{2m}\}$ with $e_i$, $1\leq i\leq m$, denoting the horizontal lift of an orthonormal frame of $M$ and with $e_{i+m}=Be_i$. Then $\pi^\estrela\alpha(e_i)=0$ and $\pi^\estrela\alpha(e_{i+m})=\pi^*\alpha(e_i)$ is a fibre constant; hence
\begin{equation} \label{coderivativeofhorizontalallift}
 \begin{split}
  -\delta\pi^*\alpha &= (\na^\mg_{e_i}\pi^*\alpha)(e_i)+
  (\na^\mg_{e_{i+m}}\pi^*\alpha)(e_{i+m}) \\
  &= e_i(\pi^*\alpha(e_i))-\pi^*\alpha(\na^\mg_{e_i}e_i)+e_{i+m}(\pi^*\alpha(e_{i+m}))-\pi^*\alpha(\na^\mg_{e_{i+m}}e_{i+m}) \\
  &= \pi^*(e_i(\alpha(e_i))-\pi^*\alpha(\na^*_{e_i}e_i+A(e_i,e_i)-\frac{1}{2}\calri(e_i,e_i))\: =\: -\pi^*\delta\alpha 
 \end{split}
\end{equation}
and
\begin{equation} \label{coderivativeofverticallift}
  -\delta\pi^\estrela\alpha=  e_i(\pi^\estrela\alpha(e_i))-\pi^\estrela\alpha(\na^\mg_{e_i}e_i)+ e_{i+m}(\pi^\estrela\alpha(e_{i+m}))-\pi^\estrela\alpha(\na^\mg_{e_{i+m}}e_{i+m}) = 0 .
\end{equation}

We also find easily, cf. \eqref{derivativeofverticalliftdual},
\begin{equation} \label{derivativeofverticallift}
 \dx(\pi^\estrela\alpha)=(\na^*_\cdot\pi^*\alpha)\wedge B^\mathrm{t}+\pi^*(\alpha\circ R^\na)S.
\end{equation}
Since $(R(X,Y)\alpha)Z=-\alpha(R(X,Y)Z)$, we have that $\pi^\estrela\alpha$ is harmonic if and only if $\na\alpha=0$.

In fact, $\pi^\estrela X^\flat=\pi^*X^\flat\circ B^\mathrm{t}=(\pi^\estrela X)^\flat$, so the conclusions for $\pi^\estrela\alpha$ and $\pi^\estrela X$ are equivalent.

One also defines the extension to $TM$ of any 1-form $\alpha\in\Omega^1_M$: in a chart $(x^i,v^i)$, with $\alpha=f_i\dx x^i$,
\begin{equation}\label{extensionform}
 \widetilde{\alpha}=v^j\papa{f_i}{x^j}\dx x^i+f_i\dx v^i=\na^*_S\pi^*\alpha+\pi^\estrela\alpha .
\end{equation}
Again, the lifts\footnote{One can test formula \eqref{derivativeofverticallift} in charts.}
\begin{equation}
 \dx x^i=\dx(x^i\circ\pi)=\pi^*\dx x^i,\qquad\quad\pi^\estrela\dx x^i=\dx v^i+v^a\Gamma_{ja}^i\dx x^j ,
\end{equation}
dual to $\pi^*\partial_i,\pi^\estrela\partial_i$, cf. \eqref{vectorlifts}, show the extension is independent of the torsion-free connection and choice of chart. $\pi^*\alpha$ is the usual pull-back. We recall $S=v^j\pi^*\partial_j$ locally.

From the first identity of \eqref{extensionform} it follows easily the equivalence $\dx\widetilde{\alpha}=0$ if and only if $\dx\alpha=0$ (this is $\partial_if_j=\partial_jf_i$).
\begin{teo}[Sasaki]
 Suppose $\alpha$ is a harmonic 1-form on $M$. Then $\widetilde{\alpha}$ is harmonic if and only if $\alpha^\sharp\lrcorner\ric{}=0$.
 
 Let $M$ be Ricci-flat and $\alpha\in\Omega^1_M$. Then $\alpha$ is harmonic if and only if $\widetilde{\alpha}$ is harmonic.
\end{teo}
\begin{proof}
We find that $\delta\alpha=-\varphi_{ab}g^{ab}$ where $\varphi_{ab}=\partial_af_b-f_i\Gamma^i_{ab}$. Due to \eqref{coderivativeofverticallift} we have $\delta\widetilde{\alpha}=\delta\beta$, where $\beta=\na^*_S\pi^*\alpha$; we easily compute that $\beta=v^p\varphi_{pk}\dx x^k$. Since $\beta^\sharp=v^p\varphi_{pj}g^{jc}\pi^*\partial_c$ is horizontal, we have  
  \begin{align*}
  -\delta\beta &= \mg(\na^\mg_{\pi^*\partial_a}\beta^\sharp,\pi^*\partial_b)g^{ab}+ \mg(\na^\mg_{\pi^\estrela\partial_a}\beta^\sharp,\pi^\estrela\partial_b)g^{ab} \\
  &= \mg(\na^*_{\pi^*\partial_a}\beta^\sharp,\pi^*\partial_b)g^{ab} \\
   &= \pi^*\partial_a(v^p\varphi_{pj}g^{jc})g_{cb}g^{ab}
   +v^p\varphi_{pj}g^{jc}\Gamma^h_{ac}g_{hb}g^{ab} \\
   &=v^p(\partial_a\varphi_{pb})g^{ba}+v^p\varphi_{pb}(\partial_ag^{ba})-v^j\Gamma^p_{aj}\varphi_{pb}g^{ba}+v^p\varphi_{pb}g^{bc}\Gamma^a_{ac} \\
   &=v^p\bigl(\partial_a\varphi_{pb}-\varphi_{ph}\Gamma^h_{ab}-\varphi_{hb}\Gamma^h_{ap}\bigr)g^{ab}
 \end{align*}
 applying $\partial_ag^{ik}=-\Gamma^k_{aj}g^{ij}-\Gamma^i_{aj}g^{kj}$. We may thus write $\delta\beta=-\tr{\na^*_\cdot\tilde{\varphi}S}$ where $\tilde{\varphi}_a^c=\varphi_{ab}g^{bc}$. However, this does not help. We must use normal coordinates and finally the hypothesis $\dx\alpha=0$ and $\delta\alpha=-\tr{\tilde{\varphi}}=0$. Then $\delta\beta=-v^p\partial_a\varphi_{pj}g^{aj}$ and therefore we find
 \begin{align*}
   \partial_a\varphi_{pj}g^{aj}  &= \bigl(\papa{^2f_j}{x^a\partial x^p}-\papa{f_i}{x^a}\Gamma^i_{pj}-f_i\papa{\Gamma^i_{pj}}{x^a}\bigr) g^{aj}  \\
   &= \papa{}{x^p}(f_i\Gamma^i_{aj}g^{aj})-f_i\papa{\Gamma^i_{pj}}{x^a}g^{aj} \\
   &= f_i\bigl(\papa{\Gamma^i_{aj}}{x^p}-\papa{\Gamma^i_{pj}}{x^a}\bigr)g^{aj} \\ &=f_iR^i_{jpa}g^{aj}  
   \end{align*}
 which proves the result. 
\end{proof}

\subsection{Infinitesimal transformations}

We recall the result of Sasaki that, for any given isometry $f$ of $M$, the induced vector bundle isomorphism $\dx f$ of $TM$, the {extension} of $f$, is a manifold isometry.

We like to see that result as a corollary of \cite[Theorem 1.3]{Alb2014d}. For it has to do naturally with the invariance of the Levi-Civita connection under isometric diffeomorphism, assured by preservation of torsion and metric identities; and, therefore, with the invariance of the horizontal distribution.

Regarding the Killing vector field equation \eqref{Killingvf}, we have the following formula from \cite{Alb2014d} for vector fields on $(TM,\mg)$. It is easily deduced from the expression for $\na^\mg$: 
\begin{equation} \label{KillingvfonTM}
  \call_X\mg(Y,Z)=\call^\na_{X^h}\pi^*g(Y,Z)+\call^\na_{X^v}\pi^\estrela
  g(Y,Z)+\mg(\calri(X,Z),Y)+\mg(\calri(X,Y),Z)
\end{equation}
where $X,Y,Z\in\XIS_{TM}$ and where $\call^\na_{X^h}\pi^*g(Y,Z):=\pi^*g(\na^*_YX^h,Z)+\pi^*g(Y,\na^*_ZX^h)$ and analogously for $\call^\na_{X^v}\pi^\estrela g$.

Thus $S$ and $\xi$ are never Killing. An interesting identity is:
\begin{equation}
 \mg(Y^v,Z^v)=\frac{1}{2}\call_\xi\mg(Y,Z).
\end{equation}

A horizontal lift $\pi^*X$ is Killing if and only if $X$ is Killing and $R(\ ,\ )X=0$. These two conditions do not yield in general any $\na$- or $\na^\mg$-parallel tensors, beside those that we may trivially conceive.

Also it is immediate that the vertical lift of a vector field $X\in\XIS_M$ is a Killing vector field on $TM$ if and only if $X$ is parallel. Indeed, taking $Y$ horizontal and $Z$ vertical,
\begin{equation}
 (\call_{\pi^\estrela X}\mg)(Y,Z)=(\call^\na_{\pi^\estrela X}\pi^\estrela g)(Y,Z) =\pi^\estrela g(\na^*_{Y}\pi^\estrela X,Z).
\end{equation}

Vertical lifts are clearly linearly independent of extended vector fields.

Sasaki swiftly observes that the extension $\widetilde{X}=X^i\partial_i+v^j\papa{X^i}{x^j}\partial_{v^i}$ of a vector field $X\in\XIS_M$ is the 1st-order part of the differential map $\dx f$ of a given diffeomorphism $f$ of $M$ and vice-versa. Indeed,
\begin{equation}
 X=\papa{}{t}\phi_t(x^i) \quad\Leftrightarrow\quad \widetilde{X}=\papa{}{t}\dx\phi_t(x^i,v^i) .
\end{equation}
Since this property is further preserved with infinitesimal isometries, the conclusion is that $\widetilde{X}$ is Killing if and only if $X$ is Killing (\cite[Corollary 3]{Sasa}).

Combining the two last results, we obtain \cite[Theorem 12]{Sasa}: if we have an invariant connected Lie subgroup $G_p$ of the $r$-dimensional isometry group $\mathrm{Isom}_0(M)$, generated by parallel vector fields $Y_1,\ldots,Y_p$, then we obtain a Lie subgroup $G_{r+p}$ of $\mathrm{Isom}_0(TM)$ of dimension $r+p$. This is due to the extensions of Killings together with the vertical lifts being an integrable distribution, that is, forming a Lie subalgebra. Indeed
\begin{equation}
 [\widetilde{X}_1,\widetilde{X}_2]=\widetilde{[X_1,X_2]},\qquad [\widetilde{X},\pi^\estrela Y]=\pi^\estrela[X,Y],\qquad [\pi^\estrela Y_1,\pi^\estrela Y_2]=0.
\end{equation}
We use both definition \eqref{extensionvfi} and the torsion-free connection $D^*$. Invariance of $G_p$ assures $\pi^\estrela[X,Y]$ belongs to the Lie subalgebra.

It is certainly worth highlighting the following result.
\begin{prop}[Sasaki]
The $\widetilde{\cdot}$ map is a Lie algebra monomorphism $\XIS_M\rr\XIS_{TM}$. Moreover, the Lie subalgebra $\fraki(M)$ of Killing vector fields on $M$ maps into the Lie subalgebra $\fraki(TM)$ of Killing vector fields on $TM$.
\end{prop}

For $X$ Killing on $M$, inserting the extension $\widetilde{X}=\pi^*X+\na^*_{S}\pi^\estrela X$ in \eqref{KillingvfonTM} yields that $\call_{\widetilde{X}}\mg(Y,Z)= \pi^\estrela g(\na^*_Y\na^*_S\pi^\estrela X,Z)+ \pi^\estrela g(Y,\na^*_Z\na^*_S\pi^\estrela X)+\mg(\calri(X,Z),Y)+\mg(\calri(X,Y),Z)$ with any $Y,Z\in TTM$. For $Y,Z$ both horizontal or both vertical, we see easily the equation vanishes on the right hand side. For $Y$ horizontal and $Z$ vertical, we find $\mg(\na^*_Y\na^*_S\pi^\estrela X,Z)+ \mg(\calri_{X,Y},Z)$. In other words, by the remark in page \pageref{remarkveryuseful}, for all $Y,Z\in\XIS_M$,
\begin{equation}\label{infinitesimalaffinetransformation}
 \na^2X(Y,Z)+R(X,Y)Z=0,
\end{equation}
which is a well-known identity satisfied by $X$ Killing, cf. \cite[p.235]{KobNomi}.

According to \cite{KobNomi,OkuTacha}, $X$ satisfying \eqref{infinitesimalaffinetransformation} is called an \textit{infinitesimal affine transformation}.

At this point follows a straightforward relation between infinitesimal transformations on $M$ and their extensions being almost-analytic. A vector field $Z$ is \textit{almost-analytic} if $\call_Z\Jnagasasa=0$, of course, presently, for the Nagano-Sasaki almost-complex structure over $TM$. We can compute via $D^*$:
\begin{equation}
 \bigl(\call_{\widetilde{X}}\Jnagasasa\bigr)(\pi^*Y)=-\pi^*R(X,Y)S-\na^*_{\pi^*Y}\na^*_S\pi^*X.
\end{equation}
It is, indeed, enough to consider horizontal lifts $\pi^*Y$. The equation is equivalent to \eqref{infinitesimalaffinetransformation}

Hence the conclusion of \cite[Theorem 2]{OkuTacha}: $X$ extends to an almost-analytic vector field if and only if $X$ is an infinitesimal affine transformation.

Let us also consider infinitesimal contact-transformations, i.e. those which leave $\theta$ invariant. The corresponding vector fields $Z$ are also said to be \textit{strictly-contact}; of course, given by $\call_Z\theta=0$.

\cite[Theorem 4]{OkuTacha} asserts that Killing extensions are strictly-contact and vice-versa. Let us see in general. Recall $\theta=S\lrcorner\mg$ and $S$ is horizontal. Then, for all $Y\in TTM$,
 \begin{equation}
 \begin{split}
    (\call_{X}\theta)(Y) &= X(\theta Y)-\theta([X,Y]) \\
    &= \mg(\na^*_{X}S,Y)+\mg(S,\na^*_{X}Y)-\mg(S,
    \na^*_{X}Y-\na^*_Y{X}) \\
    &= \mg(B^\mathrm{t}X,Y)+\mg(S,\na^*_YX^h) .
 \end{split}
 \end{equation}
 
Thus horizontal lifts are strictly-contact if and only if the base vector field is parallel. Whereas vertical lifts are never strictly-contact (though they may be symplectic).
 
Now
 \begin{equation}
 \begin{split}
    (\call_{\widetilde{X}}\theta)(Y) 
    &= \mg(B^\mathrm{t}\na^*_S\pi^\estrela X,Y)+\mg(S,\na^*_Y\pi^*X) \\
    &= \mg(\na^*_S\pi^*X,Y)+\mg(S,\na^*_Y\pi^*X)\:=\: \call_Xg(\dx\pi S,\dx\pi Y) .
 \end{split}
 \end{equation}
and we may draw the following syntheses.
\begin{teo}[\cite{OkuTacha},\cite{Sasa}]  \label{teoremameuedeles}
The following conditions on $X\in\XIS_M$ are equivalent:
\begin{meuenumerate}
 \item $X$ is Killing;
 \item $\widetilde{X}$ is Killing;
 \item $\widetilde{X}$ is strictly-contact;
 \item $\widetilde{X}$ is symplectic.
\end{meuenumerate}
\end{teo}
Condition (iv) seems to be new and will be proved later; we may recall immediately $\call_X\dx\theta=\dx\call_X\theta$.

We also find studies of Killing vector fields on $TM$ in \cite{Tanno}.

S.~Tanno discovered that if $P\in\wedge^2{T^*M}$ is a skew-symmetric and parallel (1,1)-tensor on $M$, then the vertical vector field $\pi^\estrela P\xi$ is Killing. This is trivial from \eqref{KillingvfonTM}. Indeed, we have $\na^*_Y\pi^\estrela P\xi=(\pi^\estrela P)Y^v$ and thus, relaxing the notation `$\pi^\estrela$',
\begin{equation}
 \begin{split}
 \call_{P\xi}\mg(Y,Z) &=\mg(\na^\mg_YP\xi,Z)+\mg(Y,\na^\mg_ZP\xi)   \\
&= \mg(PY^v,Z)+\mg(Y,PZ^v)+ \frac{1}{2}\mg(\calri_{Y,Z},P\xi)+
 \frac{1}{2}\mg(\calri_{Z,Y},P\xi) \:=\:0.
\end{split}
\end{equation}

The author of \cite{Tanno} claims to have found \textit{all} Killing vector fields on $TM$, since these would always decompose as $\widetilde{X_1}+\pi^\estrela P\xi+\hat{X}_2$ for some $X_1\in\XIS_M$ Killing, $P\in\End{TM}$ skew-symmetric and parallel and $X_2\in\XIS_M$ with a certain lift such that $\hat{X}_2$ is Killing. It is also deduced that $\hat{X}_2=\pi^\estrela X_2$ is vertical and therefore parallel in case $M$ is compact. However, as counter-example, with $M=\R^m$, we have the dimensions $\dim\fraki(M)=m+m(m-1)/2=m(m+1)/2$ and thus
$\dim\fraki({TM})=m(2m+1)$. We recall from \cite[p.232]{KobNomi} a Killing vector field depends on no more than its chart components and their first partial derivatives. Therefore, by Tanno's result, we would find
$\dim\fraki({TM})=m(m+1)/2+m(m-1)/2+m(m+1)/2=m(3m+1)/2$, a contradiction. With the compact flat torus we see a larger gap.

There remains, in general, an unknown set of infinitesimal isometries of the Sasaki metric.

In \cite{HedBida} and the references there-in we find the study of infinitesimal conformal transformations.

\subsection{The geodesics and the totally geodesic vector fields}

We bring forward some results on the non-vertical geodesics on $TM$. Recall the fibres are totally geodesic euclidean so they include their straight lines.

For a curve $\gamma\equiv\gamma_t$ in a trivialising neighbourhood, we may hence suppose $\gamma_t=(x^i_t,v^i_t)$ with $x$ non-constant. Then $v_t$ reads as a section of $TM\rr M$ along the curve $x_t$ in $M$, so
\begin{equation}
x^*\na_{\partial_t}v=\dot{v}^b\partial_{v^b}+\dot{x}^iv^a\Gamma_{ia}^b\partial_{v^b}=z^b\partial_{v^b},
\end{equation}
with $z^b:=\dot{v}^b+\dot{x}^iv^a\Gamma_{ia}^b$.

\cite[Theorem 1.5]{Alb2014d} yields the following system for a geodesic in $TM$, also found by Sasaki:
\begin{equation}
 \begin{cases}
  \ddot{x}^p+\dot{x}^i\dot{x}^j\Gamma_{ij}^p+\dot{x}^iz^bv^jR_{bjiq}g^{qp}=0 \\
  \dot{z}^a+\dot{x}^i\dot{z}^b\Gamma_{ib}^a=0 .
 \end{cases}
\end{equation}

Since geodesics $x_t$ in $M$ induce parallel velocities $v_t$ in $TM$ along $x_t$, and therefore all $z^b=0$, every geodesic lifts naturally to a geodesic for the Sasaki metric.

Reciprocally, a lift of a curve in $M$ is a geodesic if the curve is a geodesic of $M$. This is due to the second equation and the last paragraph of Section \ref{Section1}.

Projections by $\pi$ of geodesics yield curves in $M$ called \textit{submarine `geodesics'}. Clearly, these are geodesics if $M$ is flat.

The topology of the tangent bundle does not prevent the geodesics from \textit{not} being defined for all $t\in\R$. In other words, we are sincerely convinced $(TM,\mg)$ is a complete Riemannian manifold, as long as $M$ is complete --- cf. discussion and similar metrics referred in \cite{Alb2014d}.

Looking now to vector fields as embeddings $X:M\rr TM$, we may ask when is the embedded submanifold $X(M):=M^{^X}$ totally geodesic inside $TM$.

We have clearly $\dx X(Z)=\pi^*Z+\pi^\estrela(\na_ZX)=X_*{Z}$ with the lifts just to $T_{X_x}M^{^X}\subset T_{X_x}(TM)$, for each $x\in M$, although the lifts naturally extend to the whole space. Knowing that for $X_*{Z}_1,X_*{Z}_2$, we must have $\na^\mg_{X_*{Z}_1}X_*{Z}_2$ of the same kind of the former, we obtain the following seemingly strange equation: the submanifold $M^{^X}$ is totally geodesic if and only if, for all $Z_1,Z_2\in TM$,
\begin{equation} \label{imageofvectorfieldtotallyfeodesic}
 \na_{\na_{Z_1}Z_2+\frac{1}{2}R(X,\na_{Z_2}X)Z_1+\frac{1}{2}R(X,\na_{Z_1}X)Z_2}X= \na_{Z_1}\na_{Z_2}X-\frac{1}{2}R(Z_1,Z_2)X  .
\end{equation}

This equation is also in \cite{AbbYamp}\footnote{In this reference there is a further degree of complexity: the vector fields are considered when restricted to submanifolds of $M$. The geometry of the image is thus dependent on the two objects.}. We remark it is tensorial in $Z_1,Z_2$. Knowing the skew- and symmetric parts apparently leads to nowhere. In case we have $R(\cdot,\cdot)X=0$, as in \cite[Proposition 1.4]{Alb2014d}, then also $R(X,\cdot)\cdot=0$ by symmetry of the Riemannian curvature tensor; and the strange equation becomes
\begin{equation}\label{nablasquareparallel}
\na^2 X=0. 
\end{equation}

Finally, we have the following result with a simple proof.
\begin{teo}[{\cite[Walczak]{Walczak}}]
The vector field $X$ has constant length and the submanifold $M^{^X}$ is totally geodesic if and only if $\na X=0$.
\end{teo}
 \begin{proof}
  If $g(X,X)$ is a constant, then $g(\na X,X)=0$. We contract equation \eqref{imageofvectorfieldtotallyfeodesic} with $g(\cdot,X)$ to obtain $g(\na_{Z_1}\na_{Z_2}X,X)=0$, for every $Z_1,Z_2$. Then $g(\na_{Z_1}X,\na_{Z_1}X)=0$. The reciprocal is trivial.
\end{proof}

\subsection{Natural isometries of the Sasaki metric}

Any given constant $c$ and vector field $X\in\XIS_M$ yield immediately a vector bundle homothety followed by a translation $u\longmapsto cu+X_{\pi(u)}$. It is easy to see this map is a Sasaki metric isometry if and only if the map $Y\mapsto cY+\na_YX$ is a vector bundle isometry.

Let us see a close but quite different simple question.

Of course the maps $h:TM\lrr TM,\ h(u)=\hat{h}u$, with $\hat{h}$ smooth over $M$, are well-defined vector bundle morphisms dependent on $\hat{h}$ and the differentiable structure.

If we have a metric $g$ on $M$ and any $\cinf{}$ function $\varphi:M\rr\R$, then we may consider the conformal structure $g'=e^{2\varphi}g$ and therefore two Sasaki metrics $\mg,\mg'$ on the manifold $TM$. Since $\na'_XY=(\na+C)_XY$, where $C_XY=\dx\varphi(X)Y+\dx\varphi(Y)X-g(X,Y)\grad\varphi$, we obtain two decompositions for each vector $X=X^h+X^v=X^{h'}+X^{v'}\in TTM$.

We must consider, altogether, another $\cinf{}$ function $t:M\rr\R^+$ and the isomorphism $h:TM\lrr T'M,\ h(u)=e^{-\varphi}tu=\hat{h}u=u'$. The notation $T'M$ is for the same space $TM$ with metric $\mg'$. Then we have the following formula from \cite[Proposition 2.2]{Alb2014a}, where $\partial\varphi$ is a function, $\partial\varphi_u=\dx\varphi_{\pi(u)}(u)$, and $B,\theta,\pi^*$ are with respect to $\mg$:
\begin{equation}
 \dx h(Y)=Y^{h'}+\hat{h}\bigl(\frac{Y(t)}{t}\xi+Y^v+(\partial\varphi)BY-\theta(Y)\pi^\estrela\grad\varphi\bigr) .
\end{equation}

Finally \cite[Theorem 2.1]{Alb2014a} clarifies: the map $h$ is a homothety, i.e. $h^*\mg'=\psi\mg$ for some function $\psi$ defined on $TM$, if and only if the given functions $t$ and $\varphi$ on $M$ are constants and satisfy $e^{2\varphi}=t^2:=\psi$. In this case, $h=1$.

\subsection{Energy of vector fields on the base}

In the nineteen seventies, O.~Nouhaud and I.~Ishihara have raised the question of studying the critical points of the \textit{energy} functional
\begin{equation}
   X\longmapsto  E(X)=\frac{1}{2}\varint\|\dx X\|_\mg^2\,\vol_g
\end{equation}
on the space $\XIS_M$ of vector fields on $M$. Clearly,
\begin{equation}
 E(X)=\frac{m}{2}\vol(M)+\frac{1}{2}\varint\|\na X\|^2\,\vol_g 
\end{equation}
so the functional is essentially the second summand, the so-called vertical energy or \textit{bending} of $X$, as defined by G.~Wiegmink. By well-known Eells-Sampson theory, the critical points of that volume functional are precisely the harmonic maps. Ishihara proves the following result, which we reproduce from \cite{AbbCalvaPerrone,Calva}. The Euler-Lagrange equations evolve into the usual horizontal and vertical parts.
\begin{teo}[Ishihara]
 $X\in\XIS_M$ is a harmonic map if and only if
 \begin{equation}
  \trg{g}{R(\na_\cdot X,X)\cdot}=0 \ \ \ \ \mbox{and}\ \ \ \ \bar{\Delta}X:=\trg{g}{\na^2X}=0.
 \end{equation}
\end{teo}
For $M$ compact, by integration of $g(\bar{\Delta}X,X)$, we obtain the result which is first found by O.~Nouhaud: $X$ is a harmonic map if and only if $\na X=0$.

We have the following further remarks which we have not seen in the known literature.
\begin{coro}
  If $(\widetilde{X})^\flat$ is a closed 1-form, then $X$ is a harmonic map.
\end{coro}
\begin{proof}
 We have seen the equations for $(\widetilde{X})^\flat$ being closed imply $\na^2X=0$. Now, since $g(R(\na X,X)Z,W)=-g(R(Z,W)X,\na X)$ and $R(Z,W)X=\na^2X(Z,W)-\na^2X(W,Z)=0$, the result follows.
\end{proof}
\begin{coro}
 Let $X(M)=M^{^X}$ be a totally geodesic submanifold of $TM$. Suppose one of the following conditions is satisfied:
  i)  $R(\cdot,\cdot)X=0$, or ii)  $X$ has constant length. Then $X$ is a harmonic map.
\end{coro}
\begin{proof}
 i) follows directly from the equation deduced in \eqref{nablasquareparallel} and ii) from Walczak's Theorem.
\end{proof}
We also remark that the energy functional has been extended by O.~Gil-Medrano to the space of all immersions $M\lrr TM$, finding the same parallel condition for critical points under compactness. And has been restricted to the space $\XIS^1_M$ of unit vector fields, where harmonic maps in the previous sense are no longer critical for $E$, cf. \cite{Gil-Medrano}.

The whole theory relates to the study of the volume functional $X\mapsto\frac{1}{2}\vol_{X^*\mg}(M)$.

\subsection{Symplectic and mirror vector fields}

We aim for a complete study of all first invariants of the Sasaki metric under all canonical lifts of vector fields on $M$. Recalling the Nagano-Sasaki almost-complex structure $\Jnagasasa=B-B^\mathrm{t}$ and associated symplectic form  $\omega=-\Jnagasasa\lrcorner\mg$, we now consider \textit{symplectic} vector fields $X\in\XIS_{TM}$: such that $\call_X\omega=0$.

Using the torsion-less connection $D^*$ and the fact that $B$ and $B^\mathrm{t}$ are $\na^*$-parallel, we find, for all $Y,Z\in\XIS_{TM}$,
\begin{equation}
\begin{split}
 \call_X\omega(Y,Z) & =\dx(X\lrcorner\omega)(Y,Z)  \\
 & =-\mg(\na^*_Y\Jnagasasa X,Z)+ \mg(\na^*_Z\Jnagasasa X,Y)-\mg(\Jnagasasa X,\calri(Y,Z)) .
\end{split}
\end{equation}

Hence, for $X\in\XIS_M$, the horizontal lift is symplectic if and only if $\na^*X=0$. And the vertical lift is symplectic if and only if $X^\flat$ is closed.

Now we can prove $X$ is Killing if and only if $\widetilde{X}$ is symplectic.
\begin{proof}[Proof of statement (iv) in Theorem \ref{teoremameuedeles}.]
We have, for any $Y,Z\in TTM$,
\begin{equation*}
\begin{split}
& \call_{\widetilde{X}}\omega(Y,Z)= -\mg(\na^*_Y\pi^\estrela X,Z)+\mg(\na^*_Y\na^*_S\pi^*X,Z)+ \\
 & \hspace{27mm} +\mg(\na^*_Z\pi^\estrela X,Y)-\mg(\na^*_Z\na^*_S\pi^*X,Y)-\mg(\pi^*X,\pi^*R(Y,Z)S).
\end{split}
\end{equation*}
Taking $\pi^*Y$ horizontal and $\pi^\estrela Z$ vertical, $Y,Z\in TM$, we obtain, by the remark in page \pageref{remarkveryuseful},
\begin{equation*}
\begin{split}
  \call_{\widetilde{X}}\omega(\pi^*Y,\pi^\estrela Z) & = -\mg(\na^*_{\pi^*Y}\pi^\estrela X,\pi^\estrela Z)-\mg(\na^*_{\pi^\estrela Z}\na^*_S\pi^*X,\pi^*Y)  \\
  & = -g(\na_YX,Z)-\mg(\na^*_{\pi^*Z}\pi^*X,\pi^*Y)  \\
  & = -g(\na_YX,Z)-g(\na_ZX,Y) .
\end{split}
\end{equation*}
This proves symplectic implies Killing. The converse comes from $\call_X\dx\theta=\dx\call_X\theta$. Or we can check that the equation above for the symplectic extension vanishes for $Y,Z$ vertical, clearly, and for $Y,Z$ horizontal also. In the latter case, we must recur to the infinitesimal affine transformation equation, which holds by hypothesis, and then to Bianchi identity.
\end{proof}

Now we shall define\footnote{There are many natural non-trivial endomorphisms of $TTM$, such as $B,B^\mathrm{t},(\cdot)^h,(\cdot)^v$, so we could treat their linear combinations; but we have some special interest in discussing the present simple notion.} a vector field $X$ on $TM$ to be a $\lambda$-\textit{mirror} (respectively $\lambda$-\textit{adjoint-mirror}) vector field if $\call_XB=\lambda B$ (resp. $\call_XB^\mathrm{t}=\lambda B^\mathrm{t}$) for some constant $\lambda\in\R$.

Two interesting Lie subalgebras of vector fields are clearly defined for $\lambda=0$.

The solution space $\XIS_{TM}^\lambda$ of $\lambda$-mirror vector fields is obviously given by a particular solution $X_\lambda$ plus $\XIS_{TM}^0$. The same holds in the adjoint.

We have immediately, for all $X,Y,\in\XIS_{TM}$,
\begin{equation} \label{Liederivativemirror}
 \call_XB (Y)=[X,BY]-B[X,Y]=-\na^*_{BY}X+\na^*_YBX .
\end{equation}

$\xi$ is an example of a $-1$-mirror vector field. Every vertical lift $\pi^\estrela X$ is a 0-mirror.

For a horizontal lift $\pi^*X$, we obtain $\call_{\pi^*X}B (Y)=\na^*_Y\pi^\estrela X$. For $Z=\na^*_S\pi^\estrela X$, we obtain $\call_ZB (Y)=-\na^*_{BY}\na^*_S\pi^\estrela X=-\na^*_Y\pi^\estrela X$. Combining these two details, we find that every extension is a 0-mirror.
\begin{prop}
For any $X\in\XIS_M$, we have ${\widetilde{X}}\in\XIS_{TM}^0$, i.e.
\begin{equation}\label{Liederivativemirrorextendedvectorfield}
 \call_{\widetilde{X}} B =0.
\end{equation}
\end{prop}

Here follows the study of Euclidean space.
\begin{prop}
 Every $\lambda$-mirror vector field $X$ of $T\R^m$ decomposes uniquely as the sum of three fields: a vertical lift vector field; an extension vector field; and the horizontal lift $\lambda\pi^*P$ where $P$ is the position vector field $P=x^i\partial_i$ on $\R^m$. 
\end{prop}
\begin{proof}
We use coordinates $(x^i,v^i)$ and let $X=(A_1,A_2)=A_1^i\partial_i+A_2^j\partial_{v^j}$ be tangent to the tangent manifold, with obvious decomposition. In particular, $BX=A_1^i\partial_{v^i},\ B\partial_i=\partial_{v^i},\ B\partial_{v^j}=0$. For $\partial_{v^j}$, we get $\call_XB(\partial_{v^j})=\partial_{v^j}A_1^i\partial_{v^i}$ equal to $\lambda B\partial_{v^j}=0$ if and only if $\partial_{v^j}A^i_1=0$, for all $1\leq i,j\leq m$. For $\partial_i$, we get $\call_XB(\partial_i)=-\partial_{v^i}(A_1,A_2)+\partial_iA_1^j\partial_{v^j}$. This equals $\lambda B\partial_i=\lambda\partial_{v^i}$ if and only if $\partial_{v^i}A_1^j=0$ and $-\partial_{v^i}A_2^j+\partial_{i}A_1^j=\lambda\delta_i^j$. Now, differentiating this last equation under $\partial_{v^k}$, yields $\partial^2_{v^k,v^i}A_2^j=0$ and so we may write $A_2^j=A_{20}^j+A_{2i}^jv^i$ with $A_{20}^j,A_{2i}^j$ functions of $x$ only. Therefore, $\partial_iA_1^j=\lambda\delta_i^j+A_{2i}^j$. In sum, $A_1^j$ is a function of $x$ only and $A_2^j=A_{20}^j+(\papa{A_1^j}{x^i}-\lambda\delta_i^j)v^i$. Clearly, $A_{20}^j\partial_{v^j}$ is the desired vertical lift, $A_1^j\partial_j+\papa{A_1^j}{x^i}v^i\partial_{v^j}$ is an extension vector field and $\lambda x^i\partial_i$ is the particular solution of the $\lambda$-mirror equation.
\end{proof}

Now let us see the case of the adjoint-mirror map, in the general picture. First, it is easy to see, for all $X,Y,\in\XIS_{TM}$,
\begin{equation}\label{Liederivativeadjointmirror}
 \call_XB^\mathrm{t}(Y)=-\na^*_{B^\mathrm{t}Y}X+\na^*_YB^\mathrm{t}X +\pi^*R(X,Y)S - \calri(X,B^\mathrm{t}Y)   .
\end{equation}
\begin{prop}
 If $X\in\XIS_{TM}$ is 0-adjoint-mirror, then it is 0-mirror.
\end{prop}
\begin{proof}
 The horizontal and vertical parts of $\call_XB=0$ are $\na^*_{Y^v}X^h=0$ and $-\na^*_{BY}X^v+B\na^*_YX^h=0$. Using $Y$ vertical, the latter implies the former equation. Now let us take $Y$ vertical in $\call_XB^\mathrm{t}=0$, knowing \eqref{Liederivativeadjointmirror}, and read the horizontal part. It says $-\na^*_{B^\mathrm{t}Y^v}X^h+B^\mathrm{t}\na^*_{Y^v}X^v=0$. Applying $B$ on the left and calling $B^\mathrm{t}Y=Y^h$, we obtain $-B\na^*_{Y^h}X^h+\na^*_{BY}X^v=0$, which is the unique defining equation of 0-mirror vector field.
\end{proof}

\begin{prop}
 For $X\in\XIS_M$, both the horizontal lift \emph{or} the vertical lift are 0-adjoint-mirror if and only if $X$ is parallel.
 
 The extension $\widetilde{X}$ is 0-adjoint-mirror if and only if $X$ is an infinitesimal affine transformation.
\end{prop}
\begin{proof}
 We just see the less trivial case of $\widetilde{X}$. For $Y$ horizontal, the vanishing of \eqref{Liederivativeadjointmirror} becomes the vanishing of $\na^*_Y\na^*_S\pi^*X+\pi^*R(\pi^*X,Y)S=\pi^*\na^2X(Y,S)+\pi^*R(X,Y)S$. For $Y$ vertical, we find $-\na^*_{B^\mathrm{t}Y}\pi^*X-\na^*_{B^\mathrm{t}Y}\na^*_S\pi^\estrela X+\na^*_Y\na^*_S\pi^*X-\calri(X,B^\mathrm{t}Y)=-\pi^\estrela\na^2X(B^\mathrm{t}Y,S)-\calri(X,B^\mathrm{t}Y)$, which repeats the infinitesimal affine transformation equation.
\end{proof}
Of course, knowing \eqref{Liederivativemirrorextendedvectorfield}, we have just seen the proof that an extended vector field $\widetilde{X}$ being 0-adjoint-mirror, almost-analytic or an infinitesimal affine transformation is the same.

\begin{prop}
 A vector field $X=A_1^i\partial_i+A_2^j\partial_{v^j}$ tangent to $T\R^m$ is 0-adjoint-mirror if and only if $\partial_iA_2^j=0,\ \partial_iA_1^j=\partial_{v^i}A_2^j$, for all $1\leq i,j\leq m$.
 
 An extension vector field $\widetilde{X}=A^j\partial_j+v^k\partial_kA^i\partial_{v^i}$ is 0-adjoint-mirror if and only if $\partial^2_{j,k}A^i=0$, for all $1\leq i,j,k\leq m$, this is, if and only if $A^i(x)=a_0^i+a_j^ix^j$ with constants $a_0^i,a_j^i$, for all $1\leq i\leq m$.
\end{prop}

\ 

\textsc{R. Albuquerque}\ \ \ \ \textbar\ \ \ \ 
{\texttt{rpa@uevora.pt}}

Centro de Investiga\c c\~ao em Mate\-m\'a\-ti\-ca e Aplica\c c\~oes

Rua Rom\~ao Ramalho, 59, 671-7000 \'Evora, Portugal

The research leading to these results has received funding from Funda\c c\~ao para a Ci\^encia e a Tecnologia.

\end{document}